%% file: Uniform.tex
\documentclass[a4paper,draft=true,DIV=classic,fontsize=10pt]{scrartcl}
\KOMAoptions{DIV=last}

\input{mathheadenglish}

\begin{document}
\title{Minimal Supersolutions of BSDEs with Lower Semicontinuous Generators\thanks{We 
thank Samuel Drapeau, Martin Karliczek and Anthony R\'eveillac for helpful comments and fruitful discussions.}}

\author{{\large Gregor Heyne}\thanks{heyne@math.hu-berlin.de; Humboldt-Universit\"at 
zu Berlin, Unter den Linden 6, 10099 Berlin, Germany.~Financial 
support from MATHEON project E.2 is gratefully acknowledged.}
\e\, {\large Michael Kupper}\thanks{ kupper@math.hu-berlin.de; Humboldt-Universit\"at 
zu Berlin, Unter den Linden 6, 10099 Berlin, Germany.~Financial 
support from MATHEON project E.11 is gratefully acknowledged.} 
\e\, {\large Christoph Mainberger}\thanks{mainberg@math.hu-berlin.de; Humboldt-Universit\"at 
zu Berlin, Unter den Linden 6, 10099 Berlin, Germany.~Financial 
support from the Deutsche Forschungsgemeinschaft via SFB 649 "Economic Risk"
is gratefully acknowledged.}}
\date{\today}
\maketitle

\begin{abstract}
We study the existence and uniqueness of minimal supersolutions of backward stochastic
differential equations with generators that are jointly lower
semicontinuous, bounded below by an affine function of the control variable
and satisfy a specific normalization property. 

{\bf Keywords: }Supersolutions of Backward Stochastic Differential Equations; Semimartingale Convergence; Nonlinear 
Expectations
\end{abstract}

\input{introduction}
\input{chap01}

\input{chap02}
\input{chap03}
%
\bibliographystyle{abbrvnat}
\bibliography{bibliography}
\end{document}

%% file: mathheadenglish.tex
\usepackage{times}
\setkomafont{title}{\normalfont}
\setkomafont{disposition}{\normalfont\normalcolor\bfseries}

\usepackage[english]{babel}
\usepackage[T1]{fontenc}
\usepackage[latin1]{inputenc}
\usepackage[numbers]{natbib}
\usepackage{eurosym}
\usepackage{textcomp}
\usepackage{mathrsfs}
\usepackage{amsfonts}
\usepackage{amssymb}

\usepackage{varioref}

\usepackage{enumitem}
\usepackage[usenames]{color}
\usepackage{graphicx}
\usepackage{subfigure}
\usepackage{stmaryrd}

\usepackage[intlimits]{amsmath}
\usepackage[hyperref,amsmath,thmmarks]{ntheorem}

\usepackage[pdfborder={0 0 0},breaklinks=true]{hyperref}

\theoremseparator{.}
\newtheorem{thm}{Theorem}[section]
\newtheorem{cor}[thm]{Corollary}

\newtheorem{lem}[thm]{Lemma}
\newtheorem{prop}[thm]{Proposition}

\theorembodyfont{\upshape}

\theoremsymbol{\ensuremath{\lozenge}}

\theoremsymbol{\ensuremath{\blacklozenge}}
\theoremheaderfont{\itshape}

\theoremsymbol{\ensuremath{\square}}
\theoremheaderfont{\itshape}
\theoremstyle{nonumberplain}
\newtheorem{proof}{Proof}

\numberwithin{equation}{section}

\newcommand{\norm}[1]{\left\Vert#1\right\Vert}
\newcommand{\abs}[1]{\left\vert#1\right\vert}
\newcommand{\set}[1]{\left\{#1\right\}}

\newcommand{\R}{\mathbb{R}}
\newcommand{\N}{\mathbb{N}}

\newcommand{\SE}[2]{{\mathcal E}\left(#1\right)_#2}

\newcommand{\E}[1]{{E}\left[#1\right]}

\newcommand{\EF}[2]{\E{#1 \mid \F_{#2}}}

\newcommand{\Mss}[1]{\mathcal{E}^g_{#1}(\xi)}
\newcommand{\Mssa}[1]{\mathcal{E}^{g,a}_{#1}(\xi)}
\newcommand{\Mssp}[1]{\mathcal{E}^{g,+}_{#1}(\xi)}
\newcommand{\Acs}{\mathcal{A}(\xi,g)}

\newcommand{\Acso}{\mathcal{A}^1(\xi,g)}
\newcommand{\Aacs}{\mathcal{A}^a(\xi,g)}

\newcommand{\Acsmo}{\mathcal{A}^{M,1}(\xi,g)}

\newcommand{\F}{\mathcal{F}}

\renewcommand{\mid}{\;|\;}

\DeclareMathOperator*{\esssup}{ess\,sup}
\DeclareMathOperator*{\essinf}{ess\,inf}

\newcommand{\ba}{\begin{align}}
\newcommand{\ea}{\end{align}}

\newcommand{\bas}{\begin{align*}}
\newcommand{\Als}{\end{align*}}

\newcommand{\be}{\begin{equation}}
\newcommand{\ee}{\end{equation}}

\newcommand{\bes}{\begin{equation*}}
\newcommand{\ees}{\end{equation*}}

\newcommand{\bea}{\begin{eqnarray}}
\newcommand{\eea}{\end{eqnarray}}

\newcommand{\beas}{\begin{eqnarray*}}
\newcommand{\eeas}{\end{eqnarray*}}

\newcommand{\bem}{\begin{multline}}

\newcommand{\bems}{\begin{multline*}}

\newcommand{\q}{\quad}
\newcommand{\e}{\enspace}
\newcommand{\qq}{\qquad}

\newcommand{\ang}[1]{\left<#1\right>}  
\newcommand{\brak}[1]{\left(#1\right)}    
\newcommand{\crl}[1]{\left\{#1\right\}}   
\newcommand{\edg}[1]{\left[#1\right]}     

\newcommand{\ve}{\varepsilon}

\newcommand{\Q}{\mathbb{Q}}

\renewcommand{\mid}{\,|\,}

\newcommand*{\cadlag}{c\`adl\`ag}



%% file: introduction.tex
\section{Introduction}
\addcontentsline{toc}{section}{Introduction}
\markboth{\uppercase{Introduction}}{\uppercase{Introduction}}
On a filtered probability space, the filtration of which is generated by a
$d$-dimensional Brownian motion,
we want to give conditions ensuring that the set $\Acs$, consisting of 
all \emph{supersolutions} $(Y,Z)$ of 
a backward stochastic differential equation
with \emph{terminal condition} $\xi$ and \emph{generator} $g$,
has a minimal element.
Recall that such a supersolution is a pair $(Y,Z)$ such that, for $0\le s\le t\le T$,
\[Y_s - \int_s^t g_u(Y_u,Z_u)du + \int_s^t Z_u dW_u \ge Y_t \q\text{and}\q Y_T \ge \xi\] is satisfied.%
~We call $Y$  the \emph{value process} and $Z$ the \emph{control process} 
of the supersolution $(Y,Z)$.%
~We start by considering the process $\Mss{}$ defined as
\[
\Mss{t} = \essinf \crl{Y_t \in L^0(\F_t) \,:\, (Y,Z) \in \Acs}\,,\q t\in[0,T]\,,
\]
and show that, under suitable conditions on the generator and the terminal condition, $\Mss{}$ 
is already the value process of the unique minimal supersolution, 
that is, there is a unique control process $\hat Z$ such that $(\Mss{},\hat Z)\in\Acs$.
It was recently shown in \citet{CSTP}
that, if the generator $g$ is jointly lower semicontinuous in $y$ and $z$, convex in $z$, monotone in $y$, 
and bounded below by an affine function of $z$, a unique minimal supersolution exists.%
~Their proof was based on convex combinations of a monotone decreasing sequence of \cadlag\,supermartingales converging
pointwise to $\Mss{}$ on the rationals and made use of compactness results for sequences of martingales given
in \citet{DelbSchach01}.%
~We follow a different approach to show the existence of a unique minimal supersolution.%
~Starting with the assumption that $g$ is also jointly lower semicontinuous in $y$ and $z$
and \emph{positive} and in addition satisfies a certain \emph{normalization} condition,
we find a sequence of supersolutions converging \emph{uniformly} to the \cadlag\, supermartingale $\Mssp{}$,
the right limit of $\Mss{}$.
We then use results on convergence of semimartingales given in \citet{ProtterBarlow}
to identify a unique process $\hat Z$ such that $(\Mssp{},\hat Z)\in\Acs$.
By showing that $\Mssp{}$ always stays below $\Mss{}$, we deduce $\Mssp{}=\Mss{}$
and thus $(\Mss{},\hat Z)$ is the unique minimal supersolution.
Later on, we relax the positivity assumption to that of 
boundedness below by an affine function of $z$.
Also the normalization condition will be relaxed.
Hence both, \citet{CSTP} and our work, show the existence of a unique 
minimal supersolution of BSDEs, but under mutually singular conditions on the generator.\\

Let us briefly discuss the existing literature on related problems, 
a broader discussion of which can be found in \citet{CSTP}.
Nonlinear BSDEs were first introduced
in \citet{peng01}.
They gave existence and uniqueness results for the case of Lipschitz generators and
square integrable terminal conditions.
\citet{kobylanski01} studies BSDEs with quadratic generators,
whereas \citet{DelbaenHu} consider superquadratic BSDEs with positive
generators that are convex in $z$ and independent of $y$.%
~Among the first introducing supersolutions of BSDEs were \citet[Section 2.3]{karoui01}.
Further references can also be found in \citet{PengMono}, who studies the existence and uniqueness of
minimal supersolutions under the assumption of a Lipschitz generator and square integrable terminal 
conditions.
Most recently, \citet{CheridStad} have analyzed existence and stability of 
supersolutions of BSDEs.
They consider terminal conditions which are functionals of the underlying Brownian motion and generators
that are convex in $z$ and Lipschitz in $y$ and they work with discrete time approximations of BSDEs.
Furthermore, the concept of supersolutions is closely related to Peng's $g$ and $G$-expectations, see
for instance \cite{peng03,peng04,peng05}, since the mapping $\xi\mapsto\Mss{0}$ can be seen as a nonlinear
expectation.%
~Another related field are stochastic target problems as introduced in \citet{sonertouzi01}, the solutions
of which are obtained by dynamic programming methods and can be characterized as viscosity solutions
of second order PDEs.\\

The remainder of this paper is organized as follows. 
Setting and notations are specified in Section 2. 
A precise definition of minimal supersolutions and important structural properties
of $\Mss{}$, along with the main existence theorem,
can then be found in Sections \ref{sec21} and \ref{sec22}.
Finally, possible relaxations on the assumptions imposed on the generator, as well as a generalization
to the case of arbitrary continuous local martingales, are 
discussed in Section \ref{sec23}.
 

%% file: chap01.tex
\section{Setting and Notations}\label{sec01}

We consider a filtered probability space
$(\Omega,\F, (\F_t)_{t\ge 0},P)$,
where the filtration $(\F_t)$ is generated by a
$d$-dimensional Brownian motion $W$ and is assumed to satisfy the usual conditions.%
~For some fixed time horizon $T>0$ and for all $t\in[0,T]$,
the sets of $\F_t$-measurable random
variables are denoted by $L^0(\F_t)$, where random variables
are identified in the $P$-almost sure sense.
Let furthermore denote $L^p(\F_t)$ the set of random variables in $L^0(\F_t)$ 
with finite $p$-norm,
for $p \in [ 1,+\infty]$.
Inequalities and strict inequalities between any two
random variables or processes $X^1,X^2$ are understood in the $P$-almost sure or
in the $P\otimes dt$-almost sure sense, respectively.
In particular, two \emph{\cadlag\,} processes $X^1,X^2$ satisfying $X^1=X^2$ 
are indistinguishable, compare \citep[Chapter III]{kallenberg2002}.
We denote by $\mathcal{T}$ the set of stopping times with values in $[0,T]$ and
hereby call an increasing sequence of stopping times $(\tau^n)$, such that
$P[\bigcup_{n}\set{\tau^n=T}]=1$, a \emph{localizing sequence of stopping
times}.
By $\mathcal{S}:=\mathcal{S}(\R)$ we denote the set of 
\cadlag\, progressively measurable processes $Y$ with values in $\R$.
For $p \in \left[ 1,+\infty \right[$, we further denote by
$\mathcal{H}^p$ the set of \cadlag\, local martingales $M$ with finite $\mathcal H^p$-norm on $[0,T]$, that is 
$\norm{M}_{\mathcal H^p} := E[\langle M,M \rangle_T^{p/2}]^{1/p} <\infty$.
By $\mathcal{L}^p:=\mathcal{L}^p\left( W \right)$ we denote the set of $\R^{1\times d}$-valued, progressively
measurable processes $Z$, such that $\int ZdW \in \mathcal H^p$, that is,
$\norm{Z}_{\mathcal{L}^p}:=E[(\int_{0}^{T}Z_s^2 ds)^{p/2}]^{1/p}$ is finite.
For $Z \in \mathcal{L}^p$, the stochastic integral $(\int_0^t Z_s dW_s)_{t
\in [0,T]}$ is well defined, see \citep{protter}, and is by means
of the Burkholder-Davis-Gundy inequality \citep[Theorem 48]{protter} 
a continuous martingale.
We further denote by $\mathcal{L}:=\mathcal{L}\left( W \right)$ the set of
$\R^{1\times d}$-valued, progressively measurable processes $Z$, such that
there exists a localizing sequence of stopping times $(\tau^n)$ with
$Z1_{\left[0,\tau^n  \right]} \in \mathcal{L}^1$, for all $n\in\N$.
For $Z\in \mathcal L$, the stochastic integral $\int_{}^{}ZdW$ is well defined and is a
continuous local martingale.
Furthermore, for a process $X$, let $X^*$ denote the following expression $X^*:=\sup_{t\in[0,T]}\abs{X_t}$,
by which we define the norm $\norm{X}_{\mathcal R^\infty} := \norm{X^*}_{L^\infty}$.\\

We call a \cadlag\, semimartingale $X$ a \emph{special semimartingale}, if it can
be decomposed into $X = X_0 + M + A$, where $M$ is a local martingale and $A$ a
predictable process of finite variation such that $M_0=A_0=0$.
Such a decomposition is then unique, compare for instance \citep[Chapter III, Theorem 30]{protter},
and is called the \emph{canonical} decomposition of $X$.\\

We will use \emph{normal integrands}, a concept introduced in \cite{rockaNORM}, 
as generators of BSDEs.
Throughout this paper, a normal integrand is a function $g:\Omega \times \left[ 0,T
\right]\times \R\times \R^{1\times d} \to \left] -\infty,+\infty \right]$, 
such that
\begin{itemize}
 \item  $\brak{y,z} \mapsto g\brak{\omega,t,y,z}$ is jointly
lower semicontinuous, for all $(\omega,t) \in \Omega \times \left[ 0,T \right]$\,;
 \item $\brak{\omega,t} \mapsto g\brak{\omega,t,y,z}$ is progressively measurable, for all $(y,z) \in
\R \times \R^{1\times d}$\,.
\end{itemize}
For a normal integrand $g$ and progressively measurable processes $Y,Z$, the process $g( Y,Z)$ 
is itself progressively measurable and the integral $\int g(Y,Z)ds$ 
is well defined, $P$-almost surely, under the assumption
that $+\infty-\infty=+\infty$, see \citep[Chapter 14.F]{rockafellar02}.
Finally, as long as $g\ge 0$, the lower semicontinuity yields an extended Fatou's lemma, that is,
\bes
    \int_{}^{}\liminf_{n}g_s\left(Y_s^n,Z^n_s  \right)ds\leq
\liminf_{n}\int_{}^{}g_{s}\left( Y^n_s,Z^n_s \right)ds\,,
\ees
for any sequence $\brak{(Y^n,Z^n)}$ of progressively measurable processes.

%% file: chap02.tex
\section{Minimal Supersolutions of BSDEs}\label{sec02}
\subsection{First Definitions and Structural Properties}\label{sec21}
A pair $(Y,Z)\in\mathcal S\times \mathcal L$ is called a \emph{supersolution} of a 
BSDE, if,
for $0\le s\le t\le T$, it satisfies
\be\label{eq03}
Y_s - \int_s^t g_u(Y_u,Z_u)du + \int_s^t Z_u dW_u \ge Y_t \q\text{and}\q Y_T \ge \xi\,\,,
\ee
for a normal integrand $g$ as \emph{generator} and a \emph{terminal condition} $\xi\in L^0(\F_T)$.
For a supersolution $(Y,Z)$, we call $Y$ the \emph{value process} and $Z$ its corresponding \emph{control process}.
A control process $Z$ is said to be \emph{admissible}, if the continuous local martingale $\int Z dW$ is a 
supermartingale.\\
Throughout this paper we say that a generator $g$ is 
\begin{enumerate}[label=\textsc{(pos)},leftmargin=40pt]
 \item positive, if  $g(y,z)\ge 0$, for all $(y,z) \in \R\times\R^{1\times d}$.
\end{enumerate}
\begin{enumerate}[label=\textsc{(nor)},leftmargin=40pt]
 \item normalized, if  $g_t(y,0)=0$, for all $(t,y)\in[0,T]\times\R$.
\end{enumerate}
We are now interested in the set
\be\label{eq02}
\Acs := \crl{(Y,Z) \in\mathcal S\times \mathcal L\,:\, Z \,\,\text{is admissible and \eqref{eq03} holds}}
\ee
and the process 
\be\label{eq32}
\Mss{t} = \essinf \crl{Y_t \in L^0(\F_t) \,:\, (Y,Z) \in \Acs}\,,\q t\in[0,T]\,.
\ee
For the proof of our main existence theorem we will need some auxiliary results concerning 
structural properties of $\Mss{}$ and supersolutions $(Y,Z)$ in $\Acs$.
\begin{lem}\label{lemma1}
Let $g$ be a generator satisfying \textsc{(pos)}.
Assume further that $\Acs\neq\emptyset$ and $\xi^-\in L^1(\F_T)$. 
Then $\xi\in L^1(\F_T)$ and, 
for any $(Y,Z)\in\Acs$, the control $Z$ is unique and the value process $Y$ is a supermartingale 
such that $Y_t\ge \EF{\xi}{t}$.%
~Moreover, the unique canonical decomposition of $Y$ is given by
\be\label{eq21}
Y = Y_0 + M - A\,,
\ee
where $M= \int ZdW$ and $A$ is an increasing, predictable, \cadlag\, process with  $A_0=0$.
\end{lem}
The proof of Lemma \ref{lemma1} can be found in \citep[Lemma 3.4]{CSTP}.
\begin{prop}\label{prop2}
Suppose that $\Acs\neq\emptyset$ and let $\xi\in L^0(\F_T)$ be a terminal condition such that $\xi^-\in L^1(\F_T)$.
If $g$ satisfies \textsc{(pos)}, then the process $\Mss{}$ is a supermartingale and
\be\label{eq34}
\Mss{t}\ge\Mssp{t}:= \lim_{\underset{s\in\Q}{s\downarrow t}} \Mss{s}\,,\q\text{for all $t\in[0,T)$}\,,
\ee
and $\Mssp{T}:=\xi$.
In particular, $\Mssp{}$ is a \cadlag\, supermartingale.
Furthermore, the following two pasting properties 
hold true.
\begin{enumerate}
 \item Let $(Z^n)\subset \mathcal L$ be admissible, $\sigma\in\mathcal T$, and 
$(B_n)\subset\F_\sigma$ a partition of $\Omega$.
Then the pasted process $\bar Z = Z^11_{[0,\sigma]} + \sum_{n\ge 1} Z^n1_{B_n}1_{]\sigma,T]}$ is admissible.
 \item Let $\brak{(Y^n,Z^n)}\subset\Acs$, $\sigma\in\mathcal T$ and $(B_n)\subset\F_\sigma$ as before. 
If $Y^1_{\sigma-}1_{B_n}\ge Y^n_\sigma1_{B_n}$ 
holds true for all $n\in\N$,
then $(\bar Y, \bar Z) \in \Acs$, where
\bas
\bar Y = Y^1 1_{[0,\sigma[} + \sum_{n\ge 1} Y^n1_{B_n}1_{[\sigma,T]}\q\text{and}\q
\bar Z = Z^1 1_{[0,\sigma]} + \sum_{n\ge 1} Z^n1_{B_n}1_{]\sigma,T]}\,.
\end{align*}
\end{enumerate}
\end{prop}
\begin{proof}
The proof of the part concerning the process $\Mssp{}$ can be found in 
\citep[Proposition 3.5]{CSTP}.
As to the first pasting property, let $M^n$ and $\bar M$ denote the stochastic integrals of the $Z^n$ and $\bar Z$, respectively.
First, it follows from $(Z^n)\subset \mathcal L$ and from $(B_n)$ being a partition that $\bar Z\in\mathcal L$
and that $\int_{s\vee\sigma}^{t\vee\sigma}\bar Z_udW_u=\sum1_{B_n}\int_{s\vee\sigma}^{t\vee\sigma}Z^n_udW_u$.
Now observe that the admissibility of all $Z^n$ yields
\begin{multline*}
\EF{\bar M_t-\bar M_s}{s}=\EF{M^1_{(t\wedge\sigma)\vee s}-M^1_s}{s} +
 E\Big[\sum_{n\ge1}1_{B_n}\EF{M^n_{t\vee\sigma}-M^n_{s\vee\sigma}}{s\vee\sigma}\,\big|\,\F_s\Big]\le0\,,
\end{multline*}
for $0\le s\le t\le T$.
Thus, $\bar Z$ is admissible.%
~Finally, we can approximate
$\sigma$ from below by some foretelling sequence of stopping times $(\eta_m)$\footnote{Such a sequence satisfying $\tilde \eta_m<\sigma$, for all $m\in\N$, 
always exists, since in a Brownian filtration every
stopping time is predictable, compare \citep[Corollary V.3.3]{revuz1999}.}, and then
show, analogously to \citep[Lemma 3.1]{CSTP}, that the pair $(\bar Y,\bar Z)$ satisfies inequality \eqref{eq03}
and is thus an element of $\Acs$. 
\end{proof}
\begin{prop}\label{lemma2}
Let $0=\tau_0\le\tau_1\le\tau_2\le\cdots$ be a sequence of stopping times converging
to the finite stopping time $\tau^\ast=\lim_{n\to\infty} \tau_n$. Further, let $(Y^n)$ be a sequence 
of \cadlag\, supermartingales such that  $Y^n_{\tau_n-}\ge Y^{n+1}_{\tau_n}$,
and which satisfies $Y^n1_{[\tau_{n-1},\tau_n[}\ge M1_{[\tau_{n-1},\tau_n[}$, where $M$ is a uniformly integrable martingale.
Then, for any sequence of stopping times
$\sigma_n\in[\tau_{n-1},\tau_n[$\,, the limit  $Y^\infty:=\lim_{n\to\infty} Y^n_{\sigma_n}$ exists and the process
\[\bar Y:=\sum_{n\ge 1} Y^n 1_{[\tau_{n-1},\tau_n[}+Y^\infty 1_{[\tau^\ast,\infty[} \]
is a \cadlag\, supermartingale.~Moreover, the limit $Y^\infty$ is independent 
of the approximating sequence $(Y^n_{\sigma_n})$ and, if all $Y^n$ are
continuous and $Y^n_{\tau_n}=Y^{n+1}_{\tau_n}$, for all $n\in\N$, then $\bar Y$ is continuous. 
\end{prop}
\begin{proof}
Note that $(Y^n_{\sigma_n})$ is a $({\cal F}_{\sigma_n})$-supermartingale.
Indeed, if $(\tilde \eta_m)\uparrow \tau_n$ is a foretelling sequence of stopping 
times, then, with $\eta_m:=\tilde \eta_m\vee \tau_{n-1}$,
the family $((Y^n_{\eta_m})^-)_{m\in\N}$ is uniformly integrable and we obtain
\bems E[Y^{n+1}_{\sigma_{n+1}}\mid\F_{\sigma_{n}}]=E[E[Y^{n+1}_{\sigma_{n+1}}\mid\F_{\tau_n}]\mid\F_{\sigma_{n}}]
\le E[Y^{n+1}_{\tau_n}\mid\F_{\sigma_{n}}]\le E[Y^{n}_{\tau_{n-}}\mid\F_{\sigma_{n}}]\\
\le\liminf_m E[Y^{n}_{\eta_m}\mid\F_{\sigma_{n}}]\le\liminf_m Y^{n}_{\eta_m\wedge\sigma_{n}}=Y^n_{\sigma_n}\,.
\end{multline*}
Moreover, $((Y^n_{\sigma_n})^-)$ is uniformly integrable.
Hence, the sequence $(Y^n_{\sigma_n})$ converges  by the supermartingale convergence
theorem, see \citep[Theorems V.28,29]{Dellacherie1982}, to some random variable $Y^\infty$, $P$-almost surely,
and thus $\bar Y$ is well-defined.%
~Furthermore, the limit $Y^\infty$ is independent 
of the approximating sequence $(Y^n_{\sigma_n})$.%
~Indeed, for any other sequence $(\tilde\sigma_n)$ with $\tilde\sigma_n\in[\tau_{n-1},\tau_n[$, the limit $\lim_{n} Y^n_{\tilde\sigma_n}$ exists
by the same argumentation.%
~Now $\lim_{n} Y^n_{\sigma_n}=\lim_{n} Y^n_{\tilde\sigma_n}= Y^\infty$
holds, 
since the sequence $(\hat\sigma_n)$ defined by
\[
\hat\sigma_n:=\left\{\begin{array}{l} \e\,\,\sigma_{\frac{n}{2}}\vee\tilde\sigma_{\frac{n}{2}}\e\e\e,\,\text{for $n$ even}\\
\sigma_{\frac{n+1}{2}}\wedge\tilde\sigma_{\frac{n+1}{2}}\e\,,\,\text{for $n$ odd}\end{array}\right.
\] 
satisfies $\hat\sigma_n\in[\tau_{n-1},\tau_n[$ 
and $\lim_{n} Y^n_{\hat\sigma_n}$ exists. 
Thus, all limits must coincide.
Next, we show that $\bar Y^{\sigma_n}$ is a supermartingale, for all $n\in\N$.
To this end first observe that, for all $0\le s\le t$,
\bems
\EF{\bar Y^{\sigma_n}_t - \bar Y^{\sigma_n}_s}{s} = 
\sum_{k=0}^{n-2}E\big[E\big[\bar Y^{\sigma_n}_{(\tau_{k+1}\vee s)\wedge t}-
\bar Y^{\sigma_n}_{(\tau_{k}\vee s)\wedge t}\mid\F_{(\tau_{k}\vee s)\wedge t}\big]\mid\F_{s}\big]\\ 
+ E\big[E\big[\bar Y^{\sigma_n}_{(\sigma_n\vee s)\wedge t}-
\bar Y^{\sigma_n}_{(\tau_{n-1}\vee s)\wedge t}\mid\F_{(\tau_{n-1}\vee s)\wedge t}\big]\mid\F_{s}\big]\\
+ E\big[E\big[\bar Y^{\sigma_n}_{t} - \bar Y^{\sigma_n}_{(\sigma_n\vee s)\wedge t}\mid\F_{(\sigma_n\vee s)\wedge t}\big]\mid\F_{s}\big] \,.
\end{multline*}
Note further that, for each $n\in\N$, the process $\bar Y^{\sigma_n}$ is \cadlag\, and can only jump downwards, that is, 
$\bar Y^{\sigma_n}_{t-}\ge\bar Y^{\sigma_n}_{t}$, for all $t\in\R$.
Observe to this end that, on the one hand, $\bar Y^{\sigma_n}_{\tau_k-}=Y^k_{\tau_k-}\ge Y^{k+1}_{\tau_k}=\bar Y^{\sigma_n}_{\tau_k}$,
for all $0\le k\le n-1$, by assumption, where we assumed $\tau_{k-1}<\tau_k$, without loss of generality.
On the other hand, $Y^k$ can only jump downwards. 
Indeed, as \cadlag\, 
supermartingales, all $Y^k$ can be decomposed into
$Y^k = Y^k_0 + M^k - A^k$,
by the Doob-Meyer decomposition theorem \citep[Chapter III, Theorem
13]{protter},
where $M^k$ is a local martingale and $A^k$ a predictable, increasing process
with $A^k_0=0$. 
Since in a Brownian filtration every local martingale is
continuous, the claim follows.

Thus, for all $0\le k\le n-2$, and $(\tilde \eta_m)\uparrow \tau_{k+1}$ a foretelling sequence of stopping times,
it holds with $\eta_m:=\tilde \eta_m\vee\tau_k$, 
\bems
 E\big[\bar Y^{\sigma_n}_{(\tau_{k+1}\vee s)\wedge t}-\bar Y^{\sigma_n}_{(\tau_{k}\vee s)\wedge t}\mid\F_{(\tau_{k}\vee s)\wedge t}\big]\le
E\big[\bar Y^{\sigma_n}_{((\tau_{k+1}-)\vee s)\wedge t}-\bar Y^{\sigma_n}_{(\tau_{k}\vee s)\wedge t}\mid\F_{(\tau_{k}\vee s)\wedge t}\big]\\
=E\big[\liminf_m\bar Y^{\sigma_n}_{(\eta_m\vee s)\wedge t}- \bar Y^{\sigma_n}_{(\tau_{k}\vee s)\wedge t}\mid\F_{(\tau_{k}\vee s)\wedge t}\big]\\
\qq\qq\qq\le E\big[\liminf_m Y^{k+1}_{(\eta_m\vee s)\wedge t}\mid\F_{(\tau_{k}\vee s)\wedge t}\big]- Y^{k+1}_{(\tau_{k}\vee s)\wedge t}\\
\le\liminf_m E\big[Y^{k+1}_{(\eta_m\vee s)\wedge t}\mid\F_{(\tau_{k}\vee s)\wedge t}\big]- Y^{k+1}_{(\tau_{k}\vee s)\wedge t}\le 0\,.
\end{multline*}
Moreover, $E[\,\bar Y^{\sigma_n}_{t} - \bar Y^{\sigma_n}_{(\sigma_n\vee s)\wedge t}\mid\F_{(\sigma_n\vee s)\wedge t}]=0$, 
as well as 
\bems
E\big[\bar Y^{\sigma_n}_{(\sigma_n\vee s)\wedge t}-\bar Y^{\sigma_n}_{(\tau_{n-1}\vee s)\wedge t}\mid\F_{(\tau_{n-1}\vee s)\wedge t}\big]
\le E\big[Y^{n}_{(\sigma_n\vee s)\wedge t}- Y^{n}_{(\tau_{n-1}\vee s)\wedge t}\mid\F_{(\tau_{n-1}\vee s)\wedge t}\big]\le0\,.\end{multline*}
Combining this we obtain that $\EF{\bar Y^{\sigma_n}_t}{s}\le\bar Y^{\sigma_n}_s$.
Furthermore, $\lim_{n} \bar Y^{\sigma_n}_t = \bar Y_t$, for all $t\in\R$.
Indeed, let us write $\lim_{n} \bar Y^{\sigma_n}_t= \lim_{n}\bar Y^{\sigma_n}_t1_{\{t<\tau^*\}}+\lim_{n} \bar Y^{\sigma_n}_t1_{\{t\ge\tau^*\}}$.
Then, $\lim_{n} \bar Y^{\sigma_n}_t1_{\{t\ge\tau^*\}}=\lim_{n} Y^n_{\sigma_n}1_{\{t\ge\tau^*\}}=Y^\infty1_{\{t\ge\tau^*\}}=\bar Y_t1_{\{t\ge\tau^*\}}$
and $\lim_n\bar Y_{\sigma_n\wedge t}1_{\{t<\tau^*\}}=\bar Y_t1_{\{t<\tau^*\}}$.
Hence, the claim follows.
As a consequence of Fatou's lemma it now holds that
\bes \EF{\bar Y_t}{s}
\le\liminf_{n\to\infty} \EF{\bar Y^{\sigma_n}_t}{s} 
\le \liminf_{n\to\infty}\bar Y^{\sigma_n}_s=\bar Y_s\,,\ees
since the family $((\bar Y^{\sigma_n}_t)^-)$ 
is uniformly integrable.
Hence, $\bar Y$ is a supermartingale, which 
by construction has right-continuous paths and \citep[Theorem 1.3.8]{Karatzas1991} then yields that $\bar Y$ is even \cadlag.
Finally, whenever all $Y^n$ are
continuous and $Y^n_{\tau_n}=Y^{n+1}_{\tau_n}$ holds, for all $n\in\N$,
the process $\bar Y$ is continuous per construction.
\end{proof}

\subsection{Existence and Uniqueness of Minimal Supersolutions}\label{sec22}

We are now ready to state our main existence result.
Possible relaxations of the assumptions \textsc{(pos)} and \textsc{(nor)} imposed on the generator
are discussed in Section \ref{sec23}.
Note that it is not our focus to investigate conditions assuring the crucial
assumption that $\Acs\neq\emptyset$.
See \citet{CSTP} and the references therein for further details.
\begin{thm}\label{thm31}
Let $g$ be a generator satisfying \textsc{(pos)} and \textsc{(nor)} and $\xi\in L^0(\F_T)$ be a 
terminal condition such that $\xi^-\in L^1(\F_T)$.
If $\Acs\neq\emptyset$, then there exists a unique $\hat Z \in \mathcal L$ such that
$(\Mss{},\hat Z) \in \Acs$. 
\end{thm}
\begin{proof}
\emph{Step 1: Uniform Limit and Verification.} 
Since $\Acs\neq\emptyset$, there exist $(Y^b,Z^b)\in\Acs$.
From now on we restrict our focus to supersolutions $(\bar Y,\bar Z)$ in $\Acs$ satisfying $\bar Y_0\le Y^b_0$.
Indeed, since we are only interested in minimal supersolutions, we can paste any
value process of $(Y,Z)\in\Acs$ at $\tau:=\inf\{t>0\,:\,Y^b_t>Y_t\}\wedge T$, such that 
$\bar Y := Y^b1_{[0,\tau[} + Y1_{[\tau,T]}$ satisfies $\bar Y_0\le Y^b_0$, where the corresponding control $\bar Z$
is obtained as in Proposition \ref{prop2}.\\

Assume for the beginning that we can find a sequence
$\brak{(Y^n,Z^n)}$ within $\Acs$ such that
\be\label{eq23}
\lim_{n\to\infty}\norm{Y^n - \Mssp{}}_{\mathcal R^\infty}=0\,.
\ee 
Since all $Y^n$ are \cadlag\, supermartingales, they are, by the Doob-Meyer
decomposition theorem, see \citep[Chapter III, Theorem 13]{protter}, special
semimartingales with canonical decomposition $Y^n=Y^n_0+M^n -A^n$ as in \eqref{eq21}.
The supermartingale property of  
all $\int Z^n dW$ and $\xi\in L^1(\F_T)$, compare Lemma \ref{lemma1}, imply that $\E{A^n_T} \le Y^b_0 - \E{\xi} \in L^1(\F_T)$. 
Hence, since each $A^n$ is increasing, $\sup_n E[\int_0^T |dA^n_s|]<\infty$.
As \eqref{eq23} implies in particular that $\lim_{n\to\infty}\E{(Y^n - \Mssp{})^*}=0$, it follows 
from \citep[Theorem 1 and Corollary 2]{ProtterBarlow} that $\Mssp{}$ is a special
semimartingale with canonical 
decomposition $\Mssp{} = \Mssp{0} + M - A$ and that
\bes
\lim_{n\to\infty}\norm{M^n-M}_{\mathcal H^1} = 0\e\e,\e\e
\lim_{n\to\infty}\E{(A^n-A)^*} = 0\,.
\ees
The local martingale $M$ is continuous
and allows for a representation
of the form $M=M_0 + \int \hat ZdW$, where $\hat Z\in\mathcal L$, compare \citep[Chapter IV, Theorem 43]{protter}.
Since
\bes
\E{ \brak{\int_{0}^{T}\brak{Z^n_u-\hat Z_u}^2 du }^{1/2} } \xrightarrow[n\rightarrow +\infty]{}0\,,
\ees
we have that, up to a subsequence, $(Z^n)$ converges $P\otimes dt$-almost surely to $\hat Z$
and $\lim_{n\to\infty}\int_0^tZ^ndW=\int_0^t\hat ZdW$, for
all $t\in[0,T]$, $P$-almost surely, 
due to the Burkholder-Davis-Gundy inequality.
In particular, $\lim_{n\to\infty} Z^n(\omega)=\hat Z(\omega)$, $dt$-almost surely, for almost all $\omega\in\Omega$.

In order to verify that $(\Mssp{},\hat Z) \in \Acs$, we will use the convergence obtained above.
~More precisely, for all $0\le s\le t\le T$, Fatou's lemma together with \eqref{eq23} and the lower semicontinuity of the generator
yields
\bems
\e\Mssp{s} - \int_s^t g_u(\Mssp{u},\hat Z_u)du + \int_s^t \hat Z_u dW_u \\
\ge \limsup_n\brak{ Y_s^n- \int_s^t g_u(Y^n_u,Z^n_u )du + \int_s^t Z^n_udW_u }
\ge \limsup_n Y^n_t = \Mssp{t}\,.
\end{multline*}
The above, the positivity of $g$ and $\Mssp{}\ge\EF{\xi}{\cdot}$ imply that $\int \hat Z dW \ge \EF{\xi}{\cdot} - \Mssp{0}$. 
Hence, being bounded from below by a martingale, the continuous local martingale $\int \hat ZdW$ is a supermartingale.
Thus, $\hat Z$ is admissible and $(\Mssp{},\hat Z)\in\Acs$ and therefore, by Lemma \ref{lemma1}, $\hat Z$ is unique. 
Since we know by Proposition \ref{prop2} that $\Mss{} \ge \Mssp{}$, we deduce that $\Mss{} = \Mssp{}$ by the definition of
$\Mss{}$, 
identifying $(\Mss{},\hat Z)$ as the unique minimal supersolution.\\

\emph{Step 2: A preorder on $\Acs$.} 
As to the existence of $((Y^n,Z^n))$ satisfying \eqref{eq23},
it is sufficient to show that, for arbitrary $\ve>0$,
we can find a supersolution $(Y^{\ve},Z^{\ve})$ 
satisfying
\be\label{eq22}
\norm{Y^{\ve}-\Mssp{}}_{\mathcal R^\infty} \le \ve\,.
\ee 
We define the following preorder\footnote{Note that, 
in order to apply Zorn's lemma, 
we need a partial order instead of just a preorder.
To this end we consider equivalence classes of processes.
Two supersolutions $(Y^1,Z^1),(Y^2,Z^2)\in\Acs$ are said to be equivalent, that is, $(Y^1,Z^1)\sim(Y^2,Z^2)$, 
if $(Y^1,Z^1)\preceq(Y^2,Z^2)$ and $(Y^2,Z^2)\preceq(Y^1,Z^1)$.
This means that they are equal up to their corresponding stopping time $\tau_1=\tau_2$ as in \eqref{eq35}.
This induces a partial order on the set of equivalence classes and hence the use of Zorn's lemma
is justified.}
 on $\Acs$
\be\label{eq36}
(Y^1,Z^1)\preceq(Y^2,Z^2)\e\Leftrightarrow\e \tau_1 \le \tau_2\e\,\text{and}\e\, (Y^1,Z^1)1_{[0,\tau_1[} = (Y^2,Z^2)1_{[0,\tau_1[}\,,
\ee
where, for $i=1,2$,
\be\label{eq35}
\tau_i = \inf\crl{t\ge0\,\,:\,\,Y^i_t > \Mssp{t} + \ve } \wedge T\,.
\ee
For any totally ordered chain $((Y^i,Z^i))_{i\in I}$ within $\Acs$ with corresponding stopping times $\tau_i$, 
we want to construct an upper bound.
If we consider 
\bes
\tau^*= \esssup_{i\in I} \tau_i\,,
\ees
we know by the monotonicity of the stopping times that we can find a 
monotone subsequence $(\tau_m)$ of $(\tau_i)_{i\in I}$ such that $\tau^*=\lim_{m\to\infty}\tau_m$.
In particular, $\tau^*$ is a stopping time.
Furthermore, the structure of the preorder \eqref{eq36} yields that the value processes of the 
supersolutions $((Y^m,Z^m))$ corresponding to the stopping times $(\tau_m)$ satisfy
\be\label{eq35a}
Y^{m+1}_{\tau_m}\le Y^{m+1}_{\tau_m-}=Y^{m}_{\tau_m-}\e,\e\text{for all $m\in\N$}\,,
\ee
where the inequality follows from the fact that all $Y^m$ are \cadlag\, supermartingales, see the proof of Proposition \ref{lemma2}.\\

\emph{Step 3: A candidate upper bound $(\bar Y,\bar Z)$ for the chain $((Y^i,Z^i))_{i\in I}$.}
We construct a candidate upper bound $(\bar Y,\bar Z)$ for $((Y^i,Z^i))_{i\in I}$ satisfying
$P[\tau(\bar Y)>\tau^*\mid\tau^*<T]=1$, with $\tau(\bar Y)$ as in \eqref{eq35}.

To this end, let $(\bar\sigma_n)$ be a decreasing sequence of stopping times taking values in the rationals
and converging towards $\tau^*$
from the right\footnote{
Compare \citep[Problem 2.24]{Karatzas1991}.}.
Then the stopping times $\hat \sigma_n := \bar\sigma_n\wedge \,T$ satisfy $\hat\sigma_n 
> \tau^*$ and $\hat\sigma_n\in\Q$, on $\crl{\tau^*<T}$, for all $n$ big enough.
Let us furthermore define the following stopping time
\be\label{eq37}
\bar \tau := \inf\crl{t>\tau^*\,:\, 1_{\crl{\tau^*<T}}\abs{\Mssp{\tau^*} - \Mssp{t}} > \frac{\ve}{2}}\wedge T\,.
\ee
Due to the right-continuity of $\Mssp{}$ in $\tau^*$, it follows that $\bar \tau > \tau^*$ on $\{\tau^*<T\}$.
We now set
\ba\label{eq37a}
\sigma_n := \hat\sigma_n\wedge \bar \tau\,,\q\text{for all $n\in\N$}.
\end{align}
The above stopping times still satisfy $\lim_{n\to\infty}\sigma_n=\tau^*$ and $\sigma_n>\tau^*$ 
on $\crl{\tau^*<T}$, for all $n\in\N$.
We further define the following sets
\be\label{eq37b}
A_n := \crl{\abs{\Mssp{\tau^*} - \Mss{\sigma_m}}\vee\abs{\Mssp{\tau^*} - \Mssp{\sigma_m}} 
< \frac{\ve}{8}\,, \q\text{for all $m\ge n$} }\,.
\ee 
They satisfy $A_n\subset A_{n+1}$ and $\bigcup_n A_n = \Omega$ by definition of the 
sequence $(\sigma_m)$\footnote{Since on $\{\tau^*<T\}$, $\bar\tau>\tau^*$ and $\lim_{n}\hat\sigma_n=\tau^*$ 
with $\hat\sigma_n\in\Q$, it is ensured
that there exists some $n_0\in\N$, depending on $\omega$, such that $\sigma_n$ takes values in the rationals for all $n\ge n_0$.
By definition of $\Mssp{}$ as the right-hand side limit of $\Mss{}$ on the rationals 
and due to the right-continuity of $\Mssp{}$ in $\tau^*$, both inequalities
in the definition of $A_n$ are satisfied for all $n\ge n_0$.}. 
Note further that $A_n \in \F_{\sigma_n}$
holds true by construction.
By Proposition \ref{prop2} 
we deduce\footnote{For a detailed proof, see \citep[Proposition 3.2.]{CSTP}.} that, for each $n\in\N$, there exist $(\tilde Y^n,\tilde Z^n)\in\Acs$ such 
that
\be\label{eq37c}
\tilde Y^n_{\sigma_n} \le \Mss{\sigma_n} + \frac{\ve}{8}\,.
\ee
Next we partition $\Omega$ into $B_n :=A_n\backslash A_{n-1}$, where we set $A_0:=\emptyset$ and $\tau_0:=0$, and define the candidate upper bound as 
\ba
\bar Y &= \sum_{m\ge1} Y^m1_{[\tau_{m-1},\tau_m[}
 + 1_{\crl{\tau^*<T}}
\sum_{n\ge1} 1_{B_n}\edg{ \Big(\Mssp{\tau^*} + \frac{\ve}{2}\Big)1_{[\tau^*,\sigma_n[} +  \tilde Y^n1_{[\sigma_n,T[}} \e\e,\e\e 
\bar Y_T=\xi\,,\label{eq38a}\\
\bar Z &= 
\sum_{m\ge1} Z^m1_{]\tau_{m-1},\tau_m]} + 
1_{\crl{\tau^*<T}}\sum_{n\ge1} \tilde Z^n1_{B_n}1_{]\sigma_n,T]}\,.\label{eq38b}
\end{align}
\text{}\\
\emph{Step 4: Verification of $(\bar Y,\bar Z)\in\Acs$.}
By verifying that the pair $(\bar Y,\bar Z)$ is an element of $\Acs$,
we identify $(\bar Y,\bar Z)$ as an upper bound for the chain $((Y^i,Z^i))_{i\in I}$.
Even more, $P[\tau(\bar Y)>\tau^*\mid\tau^*<T]=1$ holds true, 
since, on the set $B_n$, we have $\bar Y_t=\Mssp{\tau^*}+\frac{\ve}{2}\le\Mssp{t} + \ve$, 
for all $t\in [\tau^*,\sigma_n[$, 
due to the definition of $\bar\tau$ in \eqref{eq37}.\\ 

\emph{Step 4a: The value process $\bar Y$ is an element of $\mathcal S$.}
By construction, the only thing to show is that $\bar Y_{\tau^{*}-}$, the left limit at $\tau^*$, exists.
This follows from Proposition \ref{lemma2}, since, by means of $((Y^m,Z^m))\subset\Acs$  and $\xi\in L^1(\F_T)$, all $Y^m$  
are \cadlag\, supermartingales, see Lemma \ref{lemma1}, which are bounded from below by a uniformly integrable martingale, more precisely $Y^m\ge\EF{\xi}{\cdot}$,
for all $m\in\N$, and satisfy \eqref{eq35a}.\\

\emph{Step 4b: The control process $\bar Z$ is an element of $\mathcal L$ and admissible.}
We proceed by defining, for each $n\in\N$, the processes
$
\bar Z^n := \sum_{m=1}^nZ^m1_{]\tau_{m-1},\tau_m]} = \bar Z 1_{[0,\tau_n]}=Z^n1_{[0,\tau_n]}$ 
and $N^n := \int \bar Z^ndW = \int Z^n1_{[0,\tau_n]}dW$,
where the equalities follow from \eqref{eq36}.
Observe that $N^{n+1}1_{[0,\tau_{n}]}=N^{n}1_{[0,\tau_{n}]}$, for all $n\in\N$, 
and that \textsc{(pos)}, \eqref{eq03} and the supermartingale property of $\int Z^ndW$ imply
\be\label{eq39}
N^n1_{[\tau_{n-1},\tau_n[}\ge 1_{[\tau_{n-1},\tau_n[} (- \EF{\xi^-}{\cdot} - Y^b_0)\,. 
\ee
By means of \eqref{eq39} and since $\xi^-\in L^1(\F_T)$, with $N^\infty:=\lim_nN^n_{\tau_{n-1}}$,
the process 
\bes
N = \sum_{n\ge1} N^n1_{[\tau_{n-1},\tau_n[} + 1_{[\tau^*,T]}N^\infty
\ees
is a well-defined continuous supermartingale due to Proposition \ref{lemma2}. 
Hence we may define a localizing sequence by setting $\kappa_n:=\inf\{t\ge 0 \,:\, \abs{N_t}> n\}\wedge T$ 
and deduce that $N$ is a continuous local martingale, because
$N^{\kappa_n}$ is a uniformly integrable martingale, for all $n\in\N$.
Indeed, for each $n\in\N$ and $m\in\N$, the process $(N^m)^{\kappa_n}$,  
being a bounded stochastic integral, is a martingale. 
Moreover, the family $(N^m_{\kappa_n\wedge t})_{m\in\N}$ is uniformly integrable
and $N_{\kappa_n\wedge t}=\lim_mN^m_{\kappa_n\wedge t}$, for all $t\in[0,T]$. 
Consequently, 
$E[N^{\kappa_n}_t\,|\,\F_s]=\lim_mE[N^m_{\kappa_n\wedge t}\,|\,\F_s]
=\lim_mN^m_{\kappa_n\wedge s}=N^{\kappa_n}_s$, for all $0\le s\le t\le T$, and the claim follows.
Since the quadratic variation of a continuous local martingale is continuous and unique, see 
\citep[page 36]{Karatzas1991},
we obtain
\bes
\int_0^{\tau^*}\bar Z_u^2du = \lim_n \int_0^{\kappa_n \wedge \tau^*}\bar Z_u^2 du 
=\lim_n \langle N \rangle_{\kappa_n \wedge \tau^*} = \langle N\rangle_{\tau^*} < \infty\,.
\ees
Observe that $\sigma := \sum_{n\ge 1}1_{B_n}\sigma_n$ is an element of $\mathcal T$.
Indeed, $\{\sigma\le t\}=\bigcup_{n\ge 1}(B_n\cap\{\sigma\le t\})=\bigcup_{n\ge 1}(B_n\cap\{\sigma_n\le t\}) \in \F_t$,
for all $t\in[0,T]$, since $B_n \in \F_{\sigma_n}$.
From $\bar Z1_{]\tau^*,\sigma]}=0$ we get that
\bes
\int_0^T \bar Z_u^2du = \langle N \rangle_{\tau^*}  
+ 1_{\crl{\tau^*<T}}\sum_{n\ge 1}1_{B_n} \int_{\sigma}^T(\tilde Z^n_u)^2du < \infty\,,
\ees
since $(\tilde Z^n)\subset\mathcal L$.
Hence 
we conclude that $\bar Z \in \mathcal L$.
As for the supermartingale property of $\int \bar Z dW$, observe that 
\bes
\int_0^{t\wedge\tau^*} \bar Z_udW_u 
= \lim_{n\to\infty}\int_0^{t\wedge\tau_n} Z^n_udW_u
\ge \lim_{n\to\infty}- \EF{\xi^-}{t\wedge\tau_n} - Y^b_0 = - \EF{\xi^-}{t\wedge\tau^*} - Y^b_0\,,
\ees
where the inequality follows from \eqref{eq03} and \textsc{(pos)}. 
Being bounded from below by a martingale, we deduce by Fatou's lemma 
that $\bar Z1_{[0,\tau^*]}$ is admissible. 
Since $\bar Z1_{]\tau^*,\sigma]}=0$ and all $\tilde Z^n$ are admissible, it follows from Proposition \ref{prop2}
that $\bar Z$ is indeed admissible.\\

\emph{Step 4c: The pair $(\bar Y,\bar Z)$ is a supersolution.} 
Finally, showing that $(\bar Y,\bar Z)$ satisfy  \eqref{eq03} 
identifies $(\bar Y,\bar Z)$ as an element of $\Acs$.
Observe first that, for all $0\le s\le t\le T$ and all $m\in\N$, the expression
$\bar Y_s - \int_s^t g_u(\bar Y_u,\bar Z_u)du + \int_s^t \bar Z_u dW_u$ can
be written as
\bem 
\bar Y_s - \int_{s}^{(\tau_m\vee s)\wedge t}g_u(\bar Y_u,\bar Z_u)du +\int_{s}^{(\tau_m\vee s)\wedge t}\bar Z_udW_u\\ 
-\int_{(\tau_m\vee s)\wedge t}^{(\tau^*\vee s)\wedge t} g_u(\bar Y_u,\bar Z_u)du 
+  \int_{(\tau_m\vee s)\wedge t}^{(\tau^*\vee s)\wedge t} \bar Z_udW_u
- \int_{(\tau^*\vee s)\wedge t}^{(\sigma\vee s)\wedge t} g_u(\bar Y_u,\bar Z_u)du\\  
+ \int_{(\tau^*\vee s)\wedge t}^{(\sigma\vee s)\wedge t} \bar Z_udW_u
- \int_{(\sigma\vee s)\wedge t}^t g_u(\bar Y_u,\bar Z_u)du + \int_{(\sigma\vee s)\wedge t}^t \bar Z_udW_u\,.\label{veri1}
\end{multline} 
Now, we have that 
\be\label{veri3a}
\bar Y_s - \int_{s}^{(\tau_m\vee s)\wedge t}g_u(\bar Y_u,\bar Z_u)du 
+\int_{s}^{(\tau_m\vee s)\wedge t}\bar Z_udW_u \,\,\ge\,\, \bar Y_{(\tau_m\vee s)\wedge t}\,,
\ee
by Proposition \ref{prop2}, 
since $((Y^m,Z^m))\subset\Acs$ and $Y^m_{\tau_m-}\ge Y^{m+1}_{\tau_{m}}$, for all $m\in\N$, due to \eqref{eq35a}. 
By letting $m$ tend to infinity and noting that 
\bes
\lim_{m\to\infty}\int_{(\tau_m\vee s)\wedge t}^{(\tau^*\vee s)\wedge t}\bar Z_udW_u=0\q\text{and}\q
\lim_{m\to\infty}\int_{(\tau_m\vee s)\wedge t}^{(\tau^*\vee s)\wedge t}g_u(\bar Y_u,\bar Z_u)du=0\,,\ees  
\eqref{veri1} and \eqref{veri3a} yield that
\bem
\bar Y_s - \int_s^t g_u(\bar Y_u,\bar Z_u)du + \int_s^t \bar Z_u dW_u\\ 
\ge \bar Y_{((\tau^*-)\vee s)\wedge t} - \int_{(\tau^*\vee s)\wedge t}^{(\sigma\vee s)\wedge t} g_u(\bar Y_u,\bar Z_u)du  
+ \int_{(\tau^*\vee s)\wedge t}^{(\sigma\vee s)\wedge t} \bar Z_udW_u\\
- \int_{(\sigma\vee s)\wedge t}^t g_u(\bar Y_u,\bar Z_u)du + \int_{(\sigma\vee s)\wedge t}^t \bar Z_udW_u\,.\label{veri5}
\end{multline}
We now use that $\bar Y$ can only jump downwards at $\tau^*$.%
~Indeed, since $\bar Y$ is \cadlag, in particular $\bar Y_{\tau^{*}-}$, the left limit 
at $\tau^*$, exists and is unique, $P$-almost surely.
Furthermore, $\lim_{m\to\infty}\bar Y_{\tau_m-}=\bar Y_{\tau^*-}$ 
and thus \[\bar Y_{\tau^{*}-}=\lim_m\bar Y_{\tau_m-}=\lim_m Y^m_{\tau_m-} \ge \lim_m \Mssp{\tau_m} + \ve  = \Mssp{\tau^{*}-} + \ve
\ge \Mssp{\tau^*}+\ve > \bar Y_{\tau^*}\,.\]
The second inequality holds, since the \cadlag\,
supermartingale $\Mssp{}$ can only jump downwards, see the proof of Proposition \ref{lemma2}.
Hence, \eqref{veri5} can be further estimated by
\bem\label{veri6}
\bar Y_s - \int_s^t g_u(\bar Y_u,\bar Z_u)du + \int_s^t \bar Z_u dW_u 
\ge \bar Y_{(\tau^*\vee s)\wedge t} 
- \int_{(\sigma\vee s)\wedge t}^t g_u(\bar Y_u,\bar Z_u)du + \int_{(\sigma\vee s)\wedge t}^t \bar Z_udW_u\,,
\end{multline}
where we used that \[\int_{(\tau^*\vee s)\wedge t}^{(\sigma\vee s)\wedge t} g_u(\bar Y_u,\bar Z_u)du  
= \int_{(\tau^*\vee s)\wedge t}^{(\sigma\vee s)\wedge t} \bar Z_udW_u=0\,,\] 
due to \eqref{eq38b}, the definition of $\sigma$ and \textsc{(nor)}.
Now observe that $\bar Y_{(\tau^*\vee s)\wedge t}\ge\bar Y_{(\sigma\vee s)\wedge t}$, 
since $\bar Y1_{[\tau^*,\sigma[}=(\Mssp{\tau^*}+\frac{\ve}{2})1_{[\tau^*,\sigma[}$ and $\bar Y$ can only jump downwards at $\sigma$.
Indeed, on the set $B_n$, by means of \eqref{eq38a}, \eqref{eq37b} and \eqref{eq37c} holds   
\bems
\bar Y_{\sigma_{n-}}=\Mssp{\tau^*} + \frac{\ve}{2} = \Mssp{\tau^*} - \Mss{\sigma_n} + \Mss{\sigma_n} + \frac{\ve}{2}\\
\ge -\frac{\ve}{8} + \Mss{\sigma_n} + \frac{\ve}{2}
\ge \tilde Y^n_{\sigma_n} -\frac{\ve}{8}+\frac{\ve}{8} = \tilde Y^n_{\sigma_n}=\bar Y_{\sigma_n}\,.
\end{multline*}
Consequently, 
\bem\label{veri8}
\bar Y_s - \int_s^t g_u(\bar Y_u,\bar Z_u)du + \int_s^t \bar Z_u dW_u \\
\ge \bar Y_{(\sigma\vee s)\wedge t} 
- \int_{(\sigma\vee s)\wedge t}^t g_u(\bar Y_u,\bar Z_u)du + \int_{(\sigma\vee s)\wedge t}^t \bar Z_udW_u\ge \bar Y_t\,\,,                               
\end{multline}
where the second inequality in \eqref{veri8} follows from $((\tilde Y^n,\tilde Z^n))\subset\Acs$ and 
Proposition \ref{prop2}.\\

\emph{Step 5: The maximal element $(Y^M,Z^M)$.}
By Zorn's lemma, there exists a maximal 
element $(Y^M,Z^M)$ in $\Acs$ with respect to the preorder \eqref{eq36}, satisfying, without loss of generality,
$Y^M_T=\xi$.
By showing that the corresponding stopping time satisfies $\tau^M=T$,
we have obtained a supersolution $(Y^M,Z^M)$ satisfying $\Vert Y^M-\Mssp{}\Vert_{\mathcal R^\infty} \le \ve$,
due to the definition of $\tau^M$ in analogy to \eqref{eq35}.
Thus, choosing $Y^M = Y^\ve$ in \eqref{eq22} would finish our proof.

But on $\{\tau^M<T\}$ we could consider the chain consisting only of $(Y^M,Z^M)$ and, analogously to 
\eqref{eq38a} and \eqref{eq38b}, construct an upper bound $(\bar Y,\bar Z)$, with corresponding stopping 
time $\tau(\bar Y)$ as in \eqref{eq35}, 
satisfying $P[\tau(\bar Y)>\tau^M\mid\tau^M<T] =1$.
This yields $P[\tau^M<T]\le P[\tau(\bar Y)>\tau^M]=0$, due to 
the maximality of $\tau^M$.  
Hence we deduce that $\tau^M=T$.
\end{proof}

The techniques used in the proof of Theorem \ref{thm31} show that $\Acs$
exhibits a certain closedness under monotone limits of decreasing supersolutions.
\begin{thm}\label{thm31a}
Let $g$ be a generator satisfying \textsc{(pos)} and \textsc{(nor)} and $\xi\in L^0(\F_T)$ a 
terminal condition such that $\xi^-\in L^1(\F_T)$.
Let furthermore $((Y^n,Z^n))$ be a decreasing sequence within $\Acs$
with pointwise limit $\bar Y_t:=\lim_n Y^n_t$, for $t\in[0,T]$.
Then $\bar Y$ is a supermartingale and it holds
\be\label{eq31a1}
\bar Y_t\ge\hat Y_t:= \lim_{\underset{s\in\Q}{s\downarrow t}} \bar Y_s\q\text{for all $t\in[0,T)$}\,.
\ee
Moreover, with $\hat Y_T:=\xi$, there is a sequence $((\tilde Y^n,\tilde Z^n))\subset\Acs$ such that 
$\lim_{n}\Vert \tilde Y^n - \hat Y\Vert_{\mathcal R^\infty}=0$, 
and a unique control $\hat Z\in\mathcal L$ such that $(\hat Y,\hat Z)\in\Acs$.
\end{thm}
\begin{proof}
First, $\bar Y$ is a supermartingale by monotone convergence.
Inequality \eqref{eq31a1} is then proved analogously to \eqref{eq34} as in \citep[Proposition 3.5]{CSTP}.
The rest follows by adapting all steps in the proof of Theorem \ref{thm31} and replacing $\Mssp{}$ by $\hat Y$.
\end{proof}
In a next step, we turn our focus to the question whether it is possible to find a minimal supersolution within
$\Acs$, the associated control process $Z$ of which belongs to $\mathcal L^1$ and 
therefore $\int ZdW$
constitutes a true martingale instead of only a supermartingale.%
~To this end we consider the following subset of $\Acs$
\be
\Acso := \crl{(Y,Z) \in \Acs \,:\, Z\in\mathcal L^1}\,.
\ee
By imposing stronger assumptions on the terminal condition $\xi$, the next theorem yields the existence
of a unique minimal supersolution in $\Acso$.
\begin{thm}\label{thm32}
Assume that the generator $g$ satisfies \textsc{(pos)} and \textsc{(nor)} and
let $\xi\in L^0(\F_T)$ be a terminal condition such that $(E[\xi^-\mid\F_{\cdot}])^*\in L^1(\F_T)$.
If $\Acso\neq\emptyset$, 
then there exists a unique $\hat Z$ such that $(\Mss{},\hat Z)\in\Acso$.
\end{thm}
\begin{proof}
$\Acso\neq\emptyset$ yields that $\Acs\neq\emptyset$, because $\Acso\subseteq\Acs$.
Also, from $\brak{\EF{\xi^-}{\cdot}}^*_T\in L^1(\F_T)$ we deduce that $\xi^-\in L^1(\F_T)$.%
~Hence, Theorem \ref{thm31} yields the existence of an unique control $\hat Z$, such 
that $(\Mss{},\hat Z)\in\Acs$.%
~Verifying that $\hat Z \in \mathcal L^1$
is done as in \citep[Theorem 4.5]{CSTP}.
\end{proof}

%% file: chap03.tex
\subsection{Relaxations of the Conditions \textsc{(nor)} and \textsc{(pos)}}\label{sec23}
In this section we discuss possible relaxations of the conditions
\textsc{(nor)} and \textsc{(pos)} imposed on the generator throughout Sections \ref{sec21} and \ref{sec22}.

First, we want to replace \textsc{(nor)} by the weaker assumption \textsc{(nor')}.
We say that a generator $g$ satisfies
\begin{enumerate}[label=\textsc{(nor')},leftmargin=40pt]
\item if, for all $\tau\in\mathcal T$, there exists some stopping time $\delta>\tau$ such that
the stochastic differential equation
\be\label{eq52}
dy_s = g_s(y_s,0)ds\,,\q y_\tau=\Mssp{\tau}+\frac{\ve}{2}
\ee
admits a solution on $[\tau,\delta]$.
\end{enumerate}
By this we obtain the following corollary to Theorem \ref{thm31}.
\begin{cor}\label{cor2}
Let $g$ be a generator satisfying \textsc{(pos)} and \textsc{(nor')} and $\xi\in L^0(\F_T)$  
a terminal condition such that $\xi^-\in L^1(\F_T)$.
If $\Acs\neq\emptyset$, then
there exists a unique $\hat Z \in \mathcal L$ such that
$(\Mss{},\hat Z) \in \Acs$.
\end{cor}
\begin{proof}
We will proceed along the lines of the proof of Theorem \ref{thm31} with the focus on the required
alterations.

The first difficulty lies in the pasting at the stopping time $\tau^*$ 
within the definition of $(\bar Y,\bar Z)$ in \eqref{eq38a} and \eqref{eq38b}.
Instead of extending by a constant function,
we concatenate the value process at $\tau^*$ with the solution of the SDE \eqref{eq52}, started
at $y_{\tau^*}=\Mssp{\tau^*}+\frac{\ve}{2}$ and denoted by $y$. 
We emphasize that the zero control is maintained.

Furthermore, we have to introduce an additional stopping time and adjust $\bar\tau$ defined in
\eqref{eq37}, in order to ensure
that our constructed value process does not leave the $\ve$-neighborhood of $\Mssp{}$.
We define 
\ba
&\kappa := \inf\{t>\tau^*\,:\, 1_{\{\tau^*<T\}}\int_{\tau^*}^t g_s(y_s,0)ds > \frac{\ve}{6}\}\wedge \delta \label{nor1}\\
\text{and}\q
&\bar \tau := \inf\{t>\tau^*\,:\, 1_{\{\tau^*<T\}} |\Mssp{\tau^*} - \Mssp{t}| > \frac{\ve}{6}\}\wedge T\label{nor2}
\end{align}
and use $\bar\kappa:=\kappa\wedge\bar\tau$ within the definition of 
the sequence $(\sigma_n)$ in analogy to \eqref{eq37a}, that is, $\sigma_n = \hat\sigma_n\wedge\bar\kappa$,
for all $n\in\N$.
As before, we set $\sigma := \sum_{n\ge 1}1_{B_n}\sigma_n$.

The pasting in \eqref{eq38a} and \eqref{eq38b} is done analogously to the proof of Theorem \ref{thm31},
but now with the distinction 
that $\bar Y 1_{[\tau^*,\sigma[}= y_{\tau^*} + 1_{[\tau^*,\sigma[}\int_{\tau^*}^\cdot g_s(y_s,0)ds$.  
The definition of the stopping times $\kappa$, $\bar\tau$ and $\sigma$ implies that, on the set $B_n$, we have 
$\bar Y_t  \le\Mssp{t}+ \ve$, for all $t\in [\tau^*,\sigma_n[$.
Indeed, observe that, for $t\in [\tau^*,\sigma_n[$,
\bems
\bar Y_t = \Mssp{\tau^*}+\frac{\ve}{2} +\int_{\tau*}^{t}g_s(y_s,0)ds\le\Mssp{\tau^*}+\frac{2\ve}{3}\\
=\Mssp{\tau^*} -\Mssp{t} + \Mssp{t} +\frac{2\ve}{3}\le\Mssp{t} + \frac{\ve}{6}+\frac{2\ve}{3}<\Mssp{t} + \ve\,,
\end{multline*}
by means of \eqref{nor1} and \eqref{nor2}, together with the definition of $\sigma_n$.
Furthermore, on the set $B_n$,
\be\label{nor3}
\bar Y_{\sigma_{n-}}=\Mssp{\tau^*} + \frac{\ve}{2} + \int_{\tau*}^{\sigma_n}g_s(y_s,0)ds \ge \Mssp{\tau^*} + \frac{\ve}{2} \ge \bar Y_{\sigma_n}\,.
\ee
The first inequality in \eqref{nor3} follows from \textsc{(pos)}, whereas the second is proved analogously to the proof of Theorem \ref{thm31},
using \eqref{nor2} and the definition of $\sigma_n$.
Hence, pasting at the stopping time
$\sigma$ is in accordance with Proposition \ref{prop2}.
 
Finally, the downward jumps at $\tau^*$ and at $\sigma$, together with the zero control in between, ensure that $(\bar Y,\bar Z)$ 
satisfies \eqref{eq03}, as was shown in Step 4c of the proof of Theorem \ref{thm31}.
The rest of the proof does not need any further alterations.
\end{proof}

Also the positivity assumption \textsc{(pos)} on the generator can be relaxed to a linear bound below,
which however has to be consistent with the assumption \textsc{(nor')}.
In the following we say that a generator $g$ is 
\begin{enumerate}[label=\textsc{(lb-nor')},leftmargin=55pt]
 \item linearly bounded from below under \textsc{(nor')}, if there exist adapted measurable processes $a$ and $b$ with values
in $\R^{1\times d}$ and $\R$, respectively, such that
$g(y,z)\ge az^T -b$, for all $(y,z)\in\R\times\R^{1\times d}$, 
and
\be\label{eq50}
\frac{dP^a}{dP}= \SE{\int a dW}{T}
\ee
defines an equivalent probability measure $P^a$.
Furthermore, $\int_0^tb_sds\in L^1(P^a)$ holds for all $t\in[0,T]$, and $a$ and $b$ are such that the positive generator
defined by
\be\label{eq51}
\bar g(y,z) := g\brak{y + \int_0^\cdot b_s ds,z} - az^T -b\,,\q\text{for all $(y,z)\in\R\times\R^{1\times d}$}\,,
\ee 
satisfies \textsc{(nor')}.
\end{enumerate}
An \textsc{(lb-nor')} setting can always be reduced to a setting
with generator satisfying \textsc{(pos)} and \textsc{(nor')}, by using the change of measure \eqref{eq50} and
$\bar g$ defined in \eqref{eq51}.%
~Hence, Lemma \ref{lemma1} and Proposition \ref{prop2}, which strongly rely on the property \textsc{(pos)},
can be applied.%
~Note that for the case $b=0$, the generator $\bar g$ even satisfies \textsc{(pos)} and \textsc{(nor)}.%
~However, we need a slightly different definition of admissibility than before.%
~A control process $Z$ is said to be \emph{$a$-admissible}, if $\int ZdW^a$ is a $P^a$-supermartingale,
where $W^a=\brak{W^1-\int a^1ds,\cdots,W^d-\int a^dds}^T$ is a $P^a$-Brownian motion by Girsanov's
theorem.\\

The set
$\Aacs := \{(Y,Z) \in\mathcal S\times \mathcal L\,:\, Z \,\,\text{is $a$-admissible and \eqref{eq03} holds}\}$,
as well as the process
\[\Mssa{t} = \essinf \{Y_t \in L^0(\F_t) \,:\, (Y,Z) \in \Aacs\}\,,\q\text{for $t\in[0,T]$}\,,\]
are defined analogously to \eqref{eq02} and \eqref{eq32}, respectively.
We are now ready to state our most general result, which follows from
Corollary \ref{cor2} and \citep[Theorem 4.16]{CSTP}.
\begin{thm} 
Let $g$ be a generator satisfying \textsc{(lb-nor')} and $\xi\in L^0(\F_T)$ a terminal condition
such that $\xi^-\in L^1(P^a)$.
If in addition $\Aacs\neq\emptyset$, then there exists a unique $a$-admissible control $\hat Z$ such that $(\Mssa{},\hat Z)\in\Aacs$.
\end{thm}


\subsection{Continuous Local Martingales and Controls in $\mathcal L^1$}\label{sec04}
Under stronger integrability conditions, 
the techniques used in the proof of Theorem \ref{thm31} can be generalized to the case where
the Brownian motion $W$ appearing in 
the stochastic integral in \eqref{eq03} is replaced by a $d$-dimensional continuous local martingale $M$.
Let us assume that $M$ is adapted to a filtration $(\F_t)_{t\ge0}$, which satisfies the usual conditions and
in which all martingales are continuous and all stopping times are predictable.%
~We consider controls within the set $\mathcal L^1:=\mathcal L^1(M)$, consisting of all $\R^{1\times d}$-valued, progressively
measurable processes $Z$, such that $\int ZdM \in \mathcal H^1$.
As before, for $Z \in \mathcal{L}^1$ the stochastic integral $(\int_0^t Z_s dM_s)_{t
\in[0,T]}$ is well defined and is by means
of the Burkholder-Davis-Gundy inequality a continuous martingale.%
~A pair $(Y,Z)\in\mathcal S\times \mathcal L^1$ is now called a supersolution of a 
BSDE, if it satisfies, for $0\le s\le t\le T$,
\be\label{eq41}
Y_s - \int_s^t g_u(Y_u,Z_u)d\ang{M}_u + \int_s^t Z_u dM_u \ge Y_t \q\text{and}\q Y_T \ge \xi\,\,,
\ee
for a normal integrand $g$ as generator and a terminal condition $\xi\in L^0(\F_T)$.
We will focus on the set
\[\Acsmo:=\crl{(Y,Z) \in\mathcal S\times \mathcal L^1\,:\, \,\text{$(Y,Z)$ satisfy \eqref{eq41}}}\,.
\] 
If we assume 
$\Acsmo$ 
to be non-empty, Theorem \ref{thm31} combined with compactness results for sequences of $\mathcal H^1$-bounded
martingales given in \citet{DelbSchach01} yields 
that
\be\label{eq43}
\Mss{t} := \essinf \crl{Y_t \in L^0(\F_t) \,:\, (Y,Z) \in \Acsmo}\,,\q t\in[0,T]\,,
\ee
is the value process of
the unique minimal supersolution within $\Acsmo$. 
Note that Lemma \ref{lemma1} and Proposition \ref{prop2} 
extend to the case where $W$ is substituted by $M$.
\begin{thm}\label{thm41}
Assume that the generator $g$ satisfies \textsc{(pos)} and \textsc{(nor)} and
let $\xi\in L^0(\F_T)$ be a terminal condition such that $(E[\xi^-\mid\F_{\cdot}])^*\in L^1(\F_T)$.
If $\Acsmo\neq\emptyset$, 
then there exists a unique $\hat Z$ such that $(\Mss{},\hat Z)\in\Acsmo$. 
\end{thm}
\begin{proof}
By assumption, there is some $(Y^b,Z^b)\in\Acsmo$ and we consider,
without loss of generality, only those pairs $(Y,Z)\in\Acsmo$ satisfying $Y\le Y^b$, 
obtained by suitable pasting as in Proposition \ref{prop2}.%
~Using the techniques of the proof of Theorem \ref{thm31}, we can find a sequence $((Y^n,Z^n))\subset\Acsmo$ 
satisfying $\lim_n\Vert Y^n-\Mssp{}\Vert_{\mathcal R^\infty}=0$, in analogy to \eqref{eq23}.
Since $(\int Z^ndM)$ is uniformly bounded in $\mathcal H^1$, compare \citep[Theorem 4.5]{CSTP},
it follows from \citep[Theorem 1]{ProtterBarlow} that 
$\Mssp{}$ is a special
semimartingale 
with canonical 
decomposition $\Mssp{} = \Mssp{0} + N - A$ and that
\be\label{eq44}
\lim_{n\to\infty}\norm{\int Z^ndM-N}_{\mathcal H^1} = 0\,.
\ee
Moreover, $N\in\mathcal H^1$.
Now \citep[Theorem 1.6]{DelbSchach01} yields the existence of some $\hat Z\in\mathcal L^1$
such that $N=\int\hat ZdM$.
By means of \eqref{eq44}, $(Z^n)$ converges, up to a subsequence, $P\otimes d\ang{M}_t$-almost surely to $\hat Z$
and $\lim_{n}\int_0^tZ^ndM=\int_0^t\hat ZdM$, for
all $t\in[0,T]$, $P$-almost surely, by means of the Burkholder-Davis-Gundy inequality.
In particular, $\lim_{n\to\infty} Z^n(\omega)=\hat Z(\omega)$, $d\ang{M}_t$-almost surely, for almost all $\omega\in\Omega$.
Verifying that $(\Mssp{},\hat Z)$ satisfy \eqref{eq41} is now done analogously to Step 1 in the proof of 
Theorem \ref{thm31}, and hence we are done.
\end{proof}